\documentclass{article}
\title{Log improvement of the Prodi-Serrin criteria for Navier-Stokes equations}
\author{Chi Hin Chan, Alexis Vasseur}
\date{}

\usepackage{amsmath}
\usepackage{amssymb}
\usepackage{amsthm}
\usepackage{amsfonts}
\usepackage[all]{xy}
\newtheoremstyle{modthm}
     {15pt}
     {3ptpt}
     {\itshape}
     {}
     {\bfseries}
     {.}
     {.5em}
     {}

\newtheoremstyle{modrem}
     {15pt}
     {0pt}
     {\rmfamily}
     {}
     {\itshape}
     {.}
     {.5em}
     {}

\theoremstyle{modthm}
\newtheorem{thm}{Theorem}
\newtheorem{prop}{Proposition}[section]
\newtheorem{lem}{Lemma}[section]

\theoremstyle{modrem}

\newcommand{\Div}{\mathrm{div}}
\newcommand{\R}{\mathbb{R}}
\begin{document}

\bibliographystyle{plain}
\maketitle
\begin{center}
Department of Mathematics \\
University of Texas at Austin
\end{center}
{ \small \noindent{\bf Abstract:} This article is devoted to a Log
improvement of Prodi-Serrin criterion for global regularity to
solutions to Navier-Stokes equations in dimension 3. It is shown
that the global regularity holds under the condition that
$|u|^5/(\log(1+|u|))$ is integrable in space time variables. }

\vskip0.2cm \noindent {\bf keywords:} Navier-Stokes, regularity
criterion, a priori estimates


\section{Introduction}

In this article, we consider  the Navier-Stokes equation on
$\mathbb{R}^3$, given by
\begin{gather}\label{NS}
\partial_{t} u -\triangle u + \Div (u\otimes u) + \nabla p = 0 , \\
\Div (u) = 0, \label{incompressibility}
\end{gather}
where $u$ is a vector-valued function representing the velocity of
the fluid, and $p$ is the pressure. Note that the pressure depends
in a non local way on the velocity $u$. It can be seen as a Lagrange
multiplier associated to the incompressible condition
(\ref{incompressibility}). The  initial  value problem of the above
equation is endowed with the condition that $u(0, \cdot ) = u_{0}
\in L^2(\mathbb{R}^3)$. Leray~\cite{Leray} and Hopf~\cite{Hopf} had
already established the existence of global weak solutions for the
Navier-Stokes equation. In particular, Leray introduced a notion of
weak solutions for the Navier-Stokes equation, and  proved that, for
every given initial datum $u_{0} \in L^2(\mathbb{R}^3)$, there
exists a global weak solution $u \in L^{\infty}(0, \infty ;
L^2(\mathbb{R}^3)) \cap L^2(0, \infty ; \dot{H}^1(\mathbb{R}^3))$
verifying the Navier-Stokes equation in the sense of distribution.
From that time on, much effort has been devoted to establish the
global existence and uniqueness of smooth solutions to the
Navier-Stokes equation. Different Criteria for regularity of the
weak solutions have been proposed. The Prodi-Serrin conditions (see
Serrin \cite{Serrin}, Prodi \cite{Prodi}, and \cite{Struwe}) states
that any weak Leray-Hopf solution verifying $u\in
L^p(0,\infty;L^q(\R^3))$ with $2/p+3/q=1$, $2\leq p<\infty$, is
regular on $(0,\infty)\times\R^3$. Notice that if $p=q$, this
corresponds to $u\in L^5((0,\infty)\times\R^3)$. The limit case of
$L^\infty(0,\infty; L^3(\R^3))$ has been solved very recently by L.
Escauriaza, G. Seregin, and V. Sverak (see \cite{Esca}). Other
criterions have been later introduced, dealing with some derivatives
of the velocity. Beale Kato and Majda \cite{Beale} showed the global
regularity under  the condition that the vorticity
$\omega=\mathrm{curl}\ u$ lies in $L^\infty(0,\infty;L^1(\R^3))$
(see Kozono and Taniuchi for improvement of this result
\cite{Kozono}). Beir\~ao da Veiga show in \cite{Vaiga} that the
boundedness of $\nabla u$ in $L^p(0,\infty; L^q(\R^3))$ for
$2/p+3/q=2$, $1<p<\infty$ ensures the global regularity. In
\cite{Constantin}, Constantin and Fefferman gave a condition
involving only the direction of the vorticity. Let us also cite a
condition involving the lower bound of the pressure introduced by
Seregin and Sverak in \cite{Sereginpressure}, and conditions
involving only one of the component of $u$ (see Penel and Pokorny
\cite{Penel}, He \cite{He}, and Zhou \cite{Zhou}).

This article is devoted to the following log improvement of the
Prodi-Serrin criterion corresponding to $p=q=5$:
\begin{thm}\label{main}
Suppose that $u$ is a weak Leray-Hopf solution of the Navier-Stokes
equation (\ref{NS}) (\ref{incompressibility}) satisfying
$$
\int_{0}^{\infty}\int_{\mathbb{R}^3}\frac{|u|^5}{\log(1 +
|u|)}dx\,ds < \infty,
$$
then, $u\in C^{\infty}((0, \infty )\times\mathbb{R}^3)$.
\end{thm}

Montgomery-Smith introduced the following criterium in \cite{Smith}:
$$
\int_0^\infty\frac{\|u(t)\|^p_{L^q(\R^3)}}{1+\log^+\|u(t)\|_{L^q(\R^3)}}\,dt<\infty.
$$
Notice that the log improvement is, here, in time only. This can be
seen as a natural Gronwall type extension of the Prodi-Serrin
conditions. So we can see it as a one dimension ODE type extension.

The goal of our result is to extend this log improvement also in
$x$. For this purpose we focused on the homogeneous case $p=q=5$,
even though extension to the Prodi-Serrin range $2\leq p<\infty$
should be doable.

The proof of Theorem \ref{main} is split into two parts. The first
point is to show that for any time $t>\lambda$, the $L^\infty$ norm
of $u$ in $x$ can be bounded in a affine way by
$$
\int_0^t\int_{\R^3}|u|^6\,dx\,dt.
$$
More precisely, we will show the following Proposition:
\begin{prop}\label{third}
For every $\lambda$ satisfying $0< \lambda < 2$, there exists some
universal constant $A_{\lambda} > 0$, depending only on $\lambda$,
such that, for any solution $u$ of the Navier-Stokes equation on
$(0, \infty )\times \mathbb{R}^3$, we have $\|u(T, \cdot
)\|_{L^{\infty}(\mathbb{R}^3)} \leqslant A_{\lambda} \{ 1 +
\int_{0}^{T}\int_{\mathbb{R}^3}|u|^6dx\,ds\}$, for any $T >
\lambda$.
\end{prop}
Then Theorem \ref{main} follows from a Gronwall argument on
$\|u(t)\|_{L^\infty(\R^3)}$, since:
\begin{eqnarray*}
&&\qquad\|u(t)\|_{L^\infty(\R^3)}\leq
A_\lambda\\
&&+A_\lambda\int_\lambda^t
\|u(s)\|_{L^\infty(\R^3)}\log(1+\|u(s)\|_{L^\infty(\R^3)})\left(\int_{\R^3}\frac{|u(s)|^5}{\log(1+|u(s)|)}\,dx\right)\,ds
\end{eqnarray*}
and the Hypothesis gives that
$\int_{\R^3}\frac{|u(s)|^5}{\log(1+|u(s)|)}\,dx$ lies in
$L^1(0,\infty)$.

Notice that the inequality of Proposition \ref{third} needs to be
invariant by the scaling of the Navier-Stokes equation:
\begin{eqnarray}\label{scaling}
u_\varepsilon(t,x)=\varepsilon u(t_0+\varepsilon^2 t,x_0+\varepsilon
x).
\end{eqnarray}
This is why the $L^6$ norm pops up, since it has the same scaling as
that of the $L^\infty$ norm. Taking advantage of the scaling
(\ref{scaling}), Proposition \ref{third} will follow from the
following rescaled Proposition:
\begin{prop}\label{first}
There exists a universal positive constant $C^*$, such that
for any solution $u$ of the Navier-Stokes equation on $[-1,1]\times \mathbb{R}^3$ satisfying
$\|u\|_{L^6(\mathbb{R}^3\times [-1,1])} \leqslant C^*$, we have $|u| \leqslant 1$ almost
everywhere on $[-\frac{1}{2}, 1]\times \mathbb{R}^3$.
\end{prop}

The proof of proposition~\ref{first} is in the same spirit as the
proof given by A. Vasseur \cite{Vasseur}. It relies on a method
first introduced by De Giorgi to show regularity of solutions to
elliptic equations with rough diffusion
coefficients~\cite{DeGiorgi}. In this paper, the proof of
proposition~\ref{first} is established though sections 2, 3, 4 and
5.  In section 6, we will deduce proposition~\ref{third} from
proposition~\ref{first}. Finally, in the last section of this paper,
we will use the conclusion of proposition~\ref{third}, together with
the fundamental result of Serrin~\cite{Serrin}, to obtain the result
of Theorem~\ref{main}.

\section{Basic setting of the whole paper}

In order to prove proposition~\ref{first}, we would like to introduce some notation
first. Then, we will state two lemmas and one proposition which are related to the proof
of proposition~\ref{first}. So, let us fix our notation as follow.
\begin{itemize}
\item for each $k \geqslant 0$, let $Q_{k} = [T_{k} , 1]\times \mathbb{R}^3$, in which
$T_{k} = -\frac{1}{2} (1 + \frac{1}{2^k})$.

\item for each $k \geqslant 0$, let $v_{k} = \{ |u| - (1 - \frac{1}{2^k}) \}_{+}$.
\item for each $k \geqslant 0$, let $d_{k} = \frac{( 1 - \frac{1}{2^k} )}{|u|}
\chi_{\{|u| \geqslant (1- \frac{1}{2^k})\}} |\nabla |u||^{2} + \frac{v_{k}}{|u|} |\nabla
u|^2$.
\item for each $k \geqslant 0$, let $U_{k} = \frac{1}{2}\|v_{k}\|^{2}_{L^{\infty}(T_{k} , 1 ;
L^{2}(\mathbb{R}^3))} + \int_{T_{k}}^{1}\int_{\mathbb{R}^{3}} d_{k}^2 dx\,dt$.
\end{itemize}
With the above setting, we are now ready to state the lemmas and
proposition which are related to proposition~\ref{first} as follow.

\begin{lem}\label{supp}
For any solution $u$ of the Navier-Stokes equation on $[-1,1]\times
\mathbb{R}^3$ satisfying $\|u\|_{L^{6}(Q_{0})} \leqslant 1$, we have
$U_{1} \leqslant A \|u\|^{6}_{L^{6}(Q_{0})}$, in which $A$ is some
universal constant strictly greater than $1$.
\end{lem}

\begin{prop}\label{index}
There exists some universal constants $B$, $\beta > 1$, such that
for any solution $u$ of the Navier-Stokes equation on $[-1,1]\times
\mathbb{R}^3$ satisfying $\|u\|_{L^{6}(Q_{0})} \leqslant
\frac{1}{A^{\frac{1}{6}}}$, we have $U_{k} \leqslant B^k
U^{\beta}_{k-1}$, for all $k\geqslant 1$. Here, $A$ is the universal
constant appearing in Lemma~\ref{supp} .
\end{prop}

Let us first show that  Lemma \ref{supp} and Proposition \ref{index}
provide the result of Proposition \ref{first}.First  we  show that
the sequence $U_k$ converges to 0 when $k$ goes to infinity. We can
use for instance the following easy lemma (see \cite{Vasseur}):
\begin{lem}\label{Vass}
For any given constants $B$, $\beta > 1$, there exists some constant
$C^*_{0}$ such that for any sequence $\{a_{k}\}_{k\geqslant 1}$
satisfying $0 < a_{1} \leq C^*_{0}$ and $a_{k} \leqslant B^k
a^{\beta}_{k-1}$, for any $k \geqslant 1$, we have
$lim_{k\rightarrow \infty} a_{k} = 0$ .
\end{lem}
Indeed, let $B$, $\beta > 1$ be the constants occurring in
proposition~\ref{index}, and let $C^*_{0}$ be the constant
associated to $B$, $\beta$ in the sense of lemma~\ref{Vass}. Now,
take $C^* = min \{\frac{1}{A^{\frac{1}{6}}} ,
(\frac{C^*_{0}}{A})^{\frac{1}{6}}\}$, in which $A$ is the universal
constant appearing in Lemma~\ref{supp}. Then, for any solution $u$
of the Navier-Stokes system on $[-1,1]\times \mathbb{R}^3$
satisfying $\|u\|_{L^6(Q_{0})} \leqslant C^*$, we have
$\|u\|_{L^6(Q_{0})} \leqslant (\frac{1}{A})^{\frac{1}{6}}$. Hence,
proposition~\ref{index} tells us that $U_{k} \leqslant B^k
U^{\beta}_{k-1}$, for all $k\geqslant 1$ must be valid. On the other
hand, Since $\|u\|_{L^6(Q_{0})} \leqslant C^* \leqslant 1$,
Lemma~\ref{supp} also implies that $U_{1} \leqslant A
\|u\|^6_{L^6(Q_{0})} \leqslant C^*_{0}$. Hence, it follows from
Lemma~\ref{Vass} that $lim_{k\rightarrow \infty} U_{k}= 0 $.
However, since we have the inequality
\begin{displaymath}
\frac{1}{2}\int_{\mathbb{R}^3}\{|u(t,x)| - 1\}^2_{+} dx \leqslant  \frac{1}{2}sup_{t\in [-\frac{1}{2} ,
1]}\int_{\mathbb{R}^3} v_{k}^2 dx \leqslant U_{k} ,
\end{displaymath}

 for every $t\in [-\frac{1}{2} , 1]$. As a result, $lim_{k\rightarrow \infty}U_{k} = 0$ immediately implies that $|u| \leqslant 1$ almost
everywhere on $[-\frac{1}{2}, 1]\times \mathbb{R}^3$. This gives the
result of Proposition \ref{first}.

\section{proof of lemma 2.1}
In this section, we will devote our effort in proving
Lemma~\ref{supp}. Let us recall that the Navier-Stokes equation on
$(-\infty, \infty)\times \mathbb{R}^3$ is
\begin{displaymath}
\partial_{t} u -\triangle u + div(u\otimes u) + \nabla P = 0 ,
\end{displaymath}
 together with the divergence free condition $div(u) = 0$. Now, by multiplying  the above equation by the
term $\frac{v_{1}}{|u|} u$, we yield the following inequality, which is valid in the sense of distribution.
\begin{displaymath}
\partial_{t} (\frac{1}{2} v_{1}^2) + d_{1}^2 - \triangle (\frac{1}{2} v_{1}^2) + div(\frac{v_{1}^2}{2} u) +
\frac{v_{1}}{|u|} u \nabla P \leqslant 0 .
\end{displaymath}
Consider now the variables $\sigma$, $t$ with $T_{0} \leqslant
\sigma \leqslant T_{1} \leqslant t \leqslant 1$, where $T_{0} = -1$,
and $T_{1} = -\frac{1}{2}(1+\frac{1}{2})$. We mention that we have
the following , which is valid in the sense of distribution.
\begin{itemize}
\item $\int_{\sigma}^{t} \int_{\mathbb{R}^3} \partial_{t} (\frac{1}{2} v_{1}^2) dx\,ds =
\frac{1}{2}\int_{\mathbb{R}^3} v_{1}^2 (t,x) dx  - \frac{1}{2} \int_{\mathbb{R}^3} v_{1}^2 (\sigma ,x) dx$.
\item $\int_{\sigma}^{t}\int_{\mathbb{R}^3} div(\frac{v_{1}^2}{2} u) - \triangle(\frac{v_{1}^2}{2}) dx\, ds
= 0$.
\end{itemize}

Hence, by taking the integral over $[\sigma, t]\times \mathbb{R}^3$ to the above inequality, we yield the
following estimation.
\begin{equation*}
\begin{split}
\frac{1}{2} \int_{\mathbb{R}^3} v_{1}^2(t,x) dx + \int_{\sigma}^{t} \int_{\mathbb{R}^3} d_{k}^2 dx\,ds
&\leqslant \frac{1}{2} \int_{\mathbb{R}^3} v_{1}^2(\sigma ,x) dx + \int_{\sigma}^{t}|\int_{\mathbb{R}^3}
\frac{v_{1}}{|u|} u \nabla P dx| ds\\
&= \frac{1}{2}\int_{\mathbb{R}^3} v_{1}^2(\sigma ,x) dx + \int_{\sigma}^{t}|\int_{\mathbb{R}^3}P
\nabla(\frac{v_{1}}{|u|}u) dx| ds \\
&\leqslant \frac{1}{2} \int_{\mathbb{R}^3} v_{1}^2(\sigma ,x) dx + 3\int_{\sigma}^{t} \int_{\mathbb{R}^3}d_{1
}|P|\chi_{\{v_{1} > 0\}} dx\, ds\\
&\leqslant \frac{1}{2} \int_{\mathbb{R}^3} v_{1}^2 (\sigma ,x) dx + \frac{3}{2} \int_{\sigma}^{t}
\int_{\mathbb{R}^3} \alpha^2d_{1}^2 dx\,ds\\
& + \frac{3}{2}\int_{\sigma}^{t} \int_{\mathbb{R}^3} \frac{|P|^2}{\alpha^2} \chi_{ \{v_{1} > 0 \}} dx\,ds ,
\end{split}
\end{equation*}
 in which $\alpha$ can be any positive constant (In the third step of the above deduction, we
have used the nontrival fact that $|\nabla (\frac{v_{k}}{|u|} u)| \leqslant 3 d_{k}$, whose
justification will be given in the last part of Section 4).
Hence we yield the following inequality which is
valid for any $\alpha > 0 $.
\begin{equation*}
\begin{split}
\int_{\mathbb{R}^3} \frac{v_{1}^2 (t,x)}{2} dx +  \int_{\sigma}^{t}
\int_{\mathbb{R}^3} \frac{(2-3\alpha^2)d_{1}^2}{2} dx\, ds
 \leqslant \int_{\mathbb{R}^3} \frac{v_{1}^2(\sigma ,x)}{2} dx \\
  &+ \int_{\sigma}^{t} \int_{\mathbb{R}^3}\frac{3 |P|^2 \chi_{\{v_{1} > 0\}}}{2\alpha^2} dx\, ds .
\end{split}
\end{equation*}
If we choose $\alpha = (\frac{1}{2})^{\frac{1}{2}}$, then the inequality shown as above becomes

\begin{displaymath}
\frac{1}{2} \int_{\mathbb{R}^3} v_{1}^2 (t,x) dx + \frac{1}{4}
\int_{\sigma}^{t}\int_{\mathbb{R}^3} d_{1}^2 dx\, ds \leqslant \int_{\mathbb{R}^3}
\frac{v_{1}^2(\sigma ,x)}{2} dx + 3\int_{\sigma}^{t}\int_{\mathbb{R}^3}|P|^2 \chi_{\{v_{1} > 0\}}
dx\, ds .
\end{displaymath}
 By taking average over $\sigma\in [T_{0} , T_{1}]$, we can carry out the following estiamtion
 \begin{displaymath}
 \int_{\mathbb{R}^3}\frac{v_{1}^2(t,x)}{2} dx +
 \int_{T_{1}}^{t}\int_{\mathbb{R}^3}\frac{d_{1}^2}{4} dx\,ds
 \leqslant \frac{4}{2}\int_{-1}^{T_{1}}\int_{\mathbb{R}^3}v_{1}^2(\sigma ,x) +
 3\int_{-1}^{t}\int_{\mathbb{R}^3} |P|^2 \chi_{\{v_{1} > 0\}} dx\, ds .
 \end{displaymath}

 Notice that, in the above inequality, the integer $4$ appears in the first term of the right
 hand side because $\frac{1}{T_{k}-T_{k-1}} = 2^2 =4$. Now, by taking
 the $L^{\infty}$-norm over $t\in [T_{1} , 1]$, we yield
 \begin{displaymath}
 \frac{1}{4} U_{1} \leqslant 2\int_{-1}^{T_{1}}\int_{\mathbb{R}^3}v_{1}^2 dx\, ds + 3\int_{Q_{0}}
 |P|^2 \chi_{\{ v_{1} > 0\}} .
  \end{displaymath}

But, we notice that

\begin{equation*}
\begin{split}
\int_{-1}^{T_{1}}\int_{\mathbb{R}^3} v_{1}^2 dx\, ds
\leqslant \int_{Q_{0}} v_{1}^2 \chi_{\{v_{1} > 0\}} \\
& \leqslant (\int_{Q_{0}}v_{1}^6)^{\frac{1}{3}}(\int_{Q_{0}}\chi_{\{v_{1 > 0}\}})^{\frac{2}{3}}\\
& \leqslant \|u\|_{L^6(Q_{0})}^{2}(\int_{Q_{0}} \chi_{\{v_{0} > \frac{1}{2}\}})^{\frac{2}{3}}\\
& \leqslant \|u\|_{L^6(Q_{0})}^{2} (2^6 \int_{Q_{0}}v_{0}^6)^{\frac{2}{3}} = 2^4 \|u\|_{L^6(Q_{0})}^{6} .
\end{split}
\end{equation*}

On the other hand, since the pressure $P$ satisfies the equation
$-\triangle P = \sum \partial_{i} \partial_{j} (u_{i}u_{j})$.

So, by the Riesz theorem in the theory of singular integral, we have
$\|P\|_{L^3(Q_{0})} \leqslant C_{3} \|u\|_{L^6(Q_{0})}^2$, in which $C_{3}$ is some universal constant
. Hence, it follows that
\begin{equation*}
\begin{split}
\int_{Q_{0}} |P|^2 \chi_{ \{v_{1} > 0\}} \leqslant \|P\|_{L^3(Q_{0})}^{2} \| \chi_{\{v_{1} > 0\}} \|_{L^3(Q_{0})} \\
& \leqslant C_{3}^2 \|u\|_{L^6(Q_{0})}^4 \|\chi_{\{v_{0} > \frac{1}{2}\}}\|_{L^3(Q_{0})} \\
& \leqslant C_{3}^2 \|u\|_{L^6(Q_{0})}^4 (2^6 \int_{Q_{0}}v_{0}^6)^{\frac{1}{3}}\\
& = 4 C_{3}^2 \|u\|_{L^6(Q_{0})}^6 .
\end{split}
\end{equation*}

Hence it follows that

\begin{equation*}
\begin{split}
\frac{1}{4} U_{1} \leqslant 2 \int_{Q_{0}} v_{1}^2 + 3 \int_{Q_{0}} |P|^2 \chi_{\{v_{1} > 0\}} \\
& \leqslant 2^5 \|u\|_{L^6(Q_{0})}^6 + 12 C_{3}^2 \|u\|_{L^{6}(Q_{0})}^6 .
\end{split}
\end{equation*}

As a result, by taking $A = 2^7 + 48 C_{3}^2$, we can at once deduce that

\begin{equation*}
U_{1} \leqslant A \|u\|_{L^6(Q_{0})}^6 .
\end{equation*}

So, we are done in establishing Lemma~\ref{supp}

\section{Preliminaries for the proof of proposition 2.1}

\begin{lem}
There exists some constant $C > 0$, such that for any $k\geqslant 1$, and any
$F\in L^{\infty} (T_{k} , 1 ; L^2(\mathbb{R}^3))$ with $\nabla F \in L^2(Q_{k})$, we have
$\|F\|_{L^{\frac{10}{3}}(Q_{k})}\leqslant C \|F\|_{L^{\infty}(T_{k} , 1 ; L^2(\mathbb{R}^3))}^{\frac{2}{5}} \|\nabla
F\|_{L^2(Q_{k})}^{\frac{3}{5}}$.
\end{lem}

\begin{proof}
By Sobolev-embedding Theorem, there is a constant $C$, depending only on the dimension of $\mathbb{R}^3$, such that
\begin{equation*}
(\int_{\mathbb{R}^3} |F(t,x)|^6 dx)^{\frac{1}{6}} \leqslant C (\int_{\mathbb{R}^3}|\nabla F(t,x)|^2
dx)^{\frac{1}{2}} .
\end{equation*}

for any $t\in [T_{k} , 1]$, where $k\geqslant 1$, and $F$ is some function which verifies $F\in L^{\infty}(T_{k},1 ;
L^2(\mathbb{R}^3))$, and $\nabla F \in L^2(Q_{k})$. By taking the power $2$ on both sides of the above inequality and
then taking integration along the variable $t\in [T_{k}, 1]$, we yield

\begin{equation*}
\int_{T_{k}}^{1} (\int_{\mathbb{R}^3} |F|^6 dx)^{\frac{1}{3}} dt \leqslant C^2
\int_{T_{k}}^{1}\int_{\mathbb{R}^3} |\nabla F|^2 dx\,dt .
\end{equation*}

On the other hand, by Holder's inequality, we have
\begin{equation*}
\begin{split}
\|F\|_{L^{\frac{10}{3}}(Q_{k})}^{\frac{10}{3}} =
\int_{T_{k}}^{1} \int_{\mathbb{R}^3} |F|^2 |F|^{\frac{4}{3}} dx\,dt\\
& \leqslant
\int_{T_{k}}^{1} (\int_{\mathbb{R}^3} |F|^6 dx)^{\frac{1}{3}}(\int_{\mathbb{R}^3}|F|^2 dx)^{\frac{2}{3}} dt\\
& \leqslant \|F\|_{L^{\infty}(T_{k} , 1 ; L^2(\mathbb{R}^3))}^{\frac{4}{3}} \|F\|_{L^2(T_{k}, 1 ;
L^6(\mathbb{R}^3))}^2.
\end{split}
\end{equation*}

By taking the advantage that $\|F\|_{L^2(T_{k} ,1 ; L^6(\mathbb{R}^3) )} \leqslant C \|\nabla F\|_{L^2(Q_{k})}$, we yield
\begin{equation*}
\|F\|_{L(Q_{k})^{\frac{10}{3}}}^{\frac{10}{3}}
\leqslant C^2 \|F\|_{L^{\infty}(T_{k} , 1 ; L^2(\mathbb{R}^3))}^{\frac{4}{3}} \|\nabla F\|_{L^2(Q_{k})}^2 .
\end{equation*}

Hence, we have

\begin{equation*}
\|F\|_{L^{\frac{10}{3}}(Q_{k})} \leqslant C \|F\|_{L^{\infty}(T_{k} , 1 ; L^2(\mathbb{R}^3))}^{\frac{2}{5}}
\|\nabla F\|_{L^2(Q_{k})}^{\frac{3}{5}} .
\end{equation*}
so, we are done
\end{proof}

\begin{lem}\label{cheb}
For any $1 < q < \infty$, we have $\|\chi_{\{ v_{k} >0 \} }\|_{L^q(Q_{k-1})} \leqslant 2^{\frac{10k}{3q}}
C^{\frac{1}{q}}U_{k-1}^{\frac{5}{3q}}$ .
\end{lem}

\begin{proof}
First, we have to notice that $\{v_{k} > 0\}$ is a subset of $\{v_{k-1} > \frac{1}{2^k}\}$, hence we
have
\begin{equation*}
\int_{Q_{k-1}} \chi_{\{ v_{k} > 0\}} \leqslant \int_{Q_{k-1}} \chi_{\{v_{k-1} > \frac{1}{2^k}\}} \leqslant
2^{\frac{10k}{3}} \int_{Q_{k-1}} |v_{k-1}|^{\frac{10}{3}} .
\end{equation*}

By our previous Lemma, we have

\begin{equation*}
\begin{split}
\|v_{k-1}\|_{L^{\frac{10}{3}(Q_{k-1})}}^{\frac{10}{3}} \leqslant \\
&C^2 \|v_{k-1}\|_{L^{\infty}(T_{k-1},1;L^2(\mathbb{R}^3))}^{\frac{4}{3}}
\|\nabla v_{k-1} \|_{L^2(Q_{k-1})}^{2} \\
& \leqslant C^2 (U_{k-1}^{\frac{1}{2}})^{\frac{4}{3}} \|d_{k-1}\|_{L^2(Q_{k-1})}^{2} \\
& \leqslant C^2 U_{k-1}^{\frac{2}{3}} U_{k-1} \\
& = C^2 U_{k-1}^{\frac{5}{3}} .
\end{split}
\end{equation*}

So, it follows that $\int_{Q_{k-1}} \chi_{\{v_{k} > 0\}} \leqslant 2^{\frac{10k}{3}} C^2
U_{k-1}^{\frac{5k}{3}}$, and hence we have
$\|\chi_{\{  v_{k} > 0 \}}\|_{L^q(Q_{k-1})} \leqslant 2^{\frac{10k}{3q}} C^{\frac{1}{q}} U_{k-1}^{\frac{5}{3q}}$,
where C is some  universal constant. So, we are done.
\end{proof}

In the proof of Lemma~\ref{cheb}, we have used the fact that $|v_{k}| \leqslant d_{k}$, whose
justification will be given immediately in the following paragraph.\\
Before we leave this section, we also want to list out some inequalities which will often be used in the proof
of proposition 1.1 as follow:

\begin{itemize}
\item $|(1-\frac{v_{k}}{|u|})u| \leqslant 1-\frac{1}{2^k}$.
\item $\frac{v_{k}}{|u|}|\nabla u| \leqslant d_{k}$.
\item $\chi_{\{v_{k} > 0\}}|\nabla |u|| \leqslant d_{k} $.
\item $|\nabla v_{k}|\leqslant d_{k}$.
\item $|\nabla (\frac{v_{k}}{|u|}u)| \leqslant 3d_{k}$.
\end{itemize}

Now, we first want to justify the validity of
$|(1-\frac{v_{k}}{|u|})u|\leqslant 1-\frac{1}{2^k}$. In the case in
which the point $(t,x)$ satisfies $|u(t,x)|< 1-\frac{1}{2^k}$, we
have $v_{k}(t,x) = 0$, and hence it follows that
\begin{equation*}
|\{1-\frac{v_{k}(t,x)}{|u(t,x)|}\}u(t,x)| = |u(t,x)| < 1-\frac{1}{2^k} .
\end{equation*}
 In the case in which $(t,x)$ satisfies $|u(t,x)|\geqslant 1-\frac{1}{2^k}$, we have
 $v_{k}(t,x) = |u(t,x)| - (1-\frac{1}{2^k})$, and hence it follows that

\begin{equation*}
|\{ 1 - \frac{v_{k}}{|u|}\} u(t,x)| = |1-\frac{|u|- (1-\frac{1}{2^k})}{|u|}||u|= 1-\frac{1}{2^k} .
\end{equation*}

So, no matter in which case, we always have the conclusion that $|(1-\frac{v_{k}}{|u|})u|\leqslant
1-\frac{1}{2^k}$.\\
Next, according to the definition of $d_{k}^2$, we can carry out the following estimation
\begin{equation*}
d_{k}^2 \geqslant \frac{v_{k}}{|u|} |\nabla u|^2 \geqslant \{\frac{v_{k}}{|u|} |\nabla u|\}^2 .
\end{equation*}
Hence, by taking square root, it follows at once that $d_{k} \geqslant \frac{v_{k}}{|u|} |\nabla u|$.\\
We now turn our attention to the inequality
$\chi_{\{ |u|\geqslant (1-\frac{1}{2^k})\}   }|\nabla |u|| \leqslant d_{k}$. To justify
it, we recall that $|\nabla u|\geqslant |\nabla |u||$. Hence, it follows from the definition of $d_{k}^2$ that

\begin{equation*}
d_{k}^2 \geqslant \frac{1-\frac{1}{2^k}}{|u|}\chi_{\{|u|\geqslant 1-\frac{1}{2^k}\}} |\nabla |u||^2
+ \{1-\frac{1-\frac{1}{2^k}}{|u|}\}\chi_{\{|u|\geqslant 1-\frac{1}{2^k}\}}|\nabla |u||^2 .
\end{equation*}

So, by simplifying the right-hand side of the above inequality, we can deduce that
$d_{k}^2 \geqslant \chi_{\{|u|\geqslant 1-\frac{1}{2^k}\}} |\nabla|u||^2$. Hence, we have
$d_{k}\geqslant \chi_{\{|u|\geqslant 1-\frac{1}{2^k}\}} |\nabla |u||$. In addition, since it is obvious to see
that $\nabla v_{k} = \chi_{\{|u|\geqslant 1-\frac{1}{2^k}\}} \nabla |u|$, we also have the result that
$|\nabla v_{k}|\leqslant d_{k}$.\\
Finally, we want to justify the inequality that $|\nabla (\frac{v_{k}}{|u|}u)|\leqslant 3d_{k}$. So, we
notice that, by applying the product rule, we have

\begin{equation*}
\nabla (\frac{v_{k}}{|u|} u) = \nabla (v_{k})\frac{u}{|u|} + \frac{v_{k}}{|u|}\nabla u
- \frac{v_{k}}{|u|^2}u\nabla |u| .
\end{equation*}

However, since $\frac{v_{k}}{|u|}|\nabla u| \leqslant d_{k}$, and
$|\frac{v_{k}}{|u|^2}u\nabla |u|| \leqslant \chi_{\{|u|\geqslant 1-\frac{1}{2^k}\}} |\nabla |u|| \leqslant
d_{k}$, it follows at once from the above expression that $|\nabla (\frac{v_{k}}{|u|}u)|\leqslant 3d_{k}$.

\section{proof of proposition 2.1}

To begin the argument, we recall that, by multiplying the equation
$\partial_{t} u - \triangle u + div( u \otimes u) + \nabla P = 0$ on $(-\infty , \infty )\times \mathbb{R}^3$, we
yield the following inequality formally, which is indeed valid in the sense of distribution

\begin{equation*}
\partial_{t} (\frac{v_{k}^2}{2}) + d_{k}^{2} - \triangle (\frac{v_{k}^2}{2}) + div(\frac{v_{k}^2}{2} u) +
\frac{v_{k}}{|u|} u \nabla P \leqslant 0 .
\end{equation*}

Next, let us consider the variables $\sigma$ , $t$ verifying $T_{k-1} \leqslant \sigma \leqslant T_{k} \leqslant t
\leqslant 1$. Then, we have

\begin{itemize}
\item $\int_{\sigma}^{t} \int_{\mathbb{R}^3} \partial_{t} (\frac{v_{k}^2}{2}) dx\,ds =
\int_{\mathbb{R}^3} \frac{v_{k}^2(t,x)}{2} dx - \int_{\mathbb{R}^3} \frac{v_{k}^2(\sigma ,x)}{2} dx$.
\item $\int_{\sigma}^{t}\int_{\mathbb{R}^3} \triangle (\frac{v_{k}^2}{2}) dx\,ds = 0$.
\item $\int_{\sigma}^{t} \int_{\mathbb{R}^3} div (\frac{v_{k}^2}{2} u) dx\,ds = 0$.
\end{itemize}

So, it is straightforward to see that

\begin{equation*}
\int_{\mathbb{R}^3} \frac{v_{k}^2(t,x)}{2} dx + \int_{\sigma}^{t}\int_{\mathbb{R}^3}d_{k}^2dx\,ds
\leqslant \int_{\mathbb{R}^3} \frac{v_{k}^2(\sigma ,x)}{2}dx +
\int_{\sigma}^{t}    \vert \int_{\mathbb{R}^3}\frac{v_{k}}{|u|} u \nabla P dx \vert       ds ,
\end{equation*}

for any $\sigma$, $t$ satisfying $T_{k-1}\leqslant \sigma \leqslant T_{k} \leqslant t \leqslant 1$. By taking the
average over the variable $\sigma$, we yield

\begin{equation*}
\int_{\mathbb{R}^3} \frac{v_{k}^2(t,x)}{2} dx + \int_{T_{k}}^{t}\int_{\mathbb{R}^3}d_{k}^2 dx\,ds
\leqslant 2^k\int_{T_{k-1}}^{T_{k}}\int_{\mathbb{R}^3}v_{k}^2(s,x)dx\,ds
+ \int_{T_{k-1}}^{t}|\int_{\mathbb{R}^3}\frac{v_{k}}{|u|} u \nabla Pdx|ds .
\end{equation*}

By taking  the sup over $t\in [T_{k} , 1]$. the above inequality will give the following

\begin{equation*}
U_{k} \leqslant 2^k\int_{Q_{k-1}}v_{k}^2 + \int_{T_{k-1}}^{1}|\int_{\mathbb{R}^3}\frac{v_{k}}{|u|}u\nabla
Pdx|ds .
\end{equation*}

But, from Lemma~\ref{cheb} and Holder's inequality, we have

\begin{equation*}
\begin{split}
\int_{Q_{k-1}}v_{k}^2 = \int_{Q_{k-1}}v_{k}^2 \chi_{\{v_{k} > 0\}} \\
& \leqslant (\int_{Q_{k-1}}v_{k}^{\frac{10}{3}})^{\frac{3}{5}} \|\chi_{\{v_{k} > 0 \}}\|_{L^{\frac{5}{2}}(Q_{k-1})}\\
& \leqslant \|v_{k}\|_{L^{\frac{10}{3}}(Q_{k-1})}^{2} 2^{\frac{4k}{3}}C^{\frac{2}{5}}U_{k-1}^{\frac{2}{3}}\\
& \leqslant \|v_{k-1}\|_{L^{\frac{10}{3}}(Q_{k-1})}^{2}2^{\frac{4k}{3}}C^{\frac{2}{5}}U_{k-1}^{\frac{2}{3}}\\
& \leqslant CU_{k-1}^{\frac{5}{3}}2^{\frac{4k}{3}} .
\end{split}
\end{equation*}

As a result, we have the following conclusion

\begin{equation}
U_{k}\leqslant 2^{\frac{7k}{3}}C U_{k-1}^{\frac{5}{3}} +
\int_{T_{k-1}}^{1}|\int_{\mathbb{R}^3}\frac{v_{k}}{|u|}u \nabla p dx|ds .
\end{equation}

Now, in order to estimate the term $\int_{T_{k-1}}^{1}|\int_{\mathbb{R}^3}\frac{v_{k}}{|u|}u\nabla Pdx|ds$, we would
like to carry out the following computation

\begin{equation*}
\begin{split}
-\triangle P = \sum \partial_{i} \partial_{j} (u_{i}u_{j})\\
& = \sum \partial_{i} \partial_{j} \{ (1-\frac{v_{k}}{|u|})u_{i} (1-\frac{v_{k}}{|u|})u_{j}
+ 2 (1-\frac{v_{k}}{|u|}) u_{i} \frac{v_{k}}{|u|} u_{j}  \}\\
& + \sum \partial_{i} \partial_{j} \{ \frac{v_{k}}{|u|} u_{i} \frac{v_{k}}{|u|} u_{j} \} .
\end{split}
\end{equation*}

This motivates us to decompose $P$ as $P = P_{k1} + P_{k2}$, in which

\begin{equation*}
-\triangle P_{k1} = \sum \partial_{i} \partial_{j} \{ (1-\frac{v_{k}}{|u|})u_{i}
 (1-\frac{v_{k}}{|u|})u_{j} + 2(1-\frac{v_{k}}{|u|})u_{i} \frac{v_{k}}{|u|}u_{j} \} ,
\end{equation*}

and that

\begin{equation*}
-\triangle P_{k2} = \sum \partial_{i} \partial_{j}\{\frac{v_{k}}{|u|}u_{i} \frac{v_{k}}{|u|}u_{j} \} .
\end{equation*}

First, we have to notice that:

\begin{equation*}
\begin{split}
| (1-\frac{v_{k}}{|u|})^2 u_{i} u_{j} + 2(1-\frac{v_{k}}{|u|})u_{i} \frac{v_{k}}{|u|} u_{j} |\\
&\leqslant (1-\frac{1}{2^k}) \{(1-\frac{v_{k}}{|u|}) |u_{j}| + 2\frac{v_{k}}{|u|}|u_{j}|\}\\
&\leqslant (1-\frac{v_{k}}{|u|})|u_{j}| + 2\frac{v_{k}}{|u|}|u_{j}|\\
&\leqslant 3|u_{j}| \leqslant 3|u| .
\end{split}
\end{equation*}

So, by Riesz's Theorem in the theroy of singular operator, we yield

\begin{equation*}
\|P_{k1}\|_{L^6(Q_{k-1})} \leqslant C_{6} \|3u\|_{L^6(Q_{k-1})} \leqslant 3C_{6} (\frac{1}{A})^{\frac{1}{6}}
\leqslant 3C_{6} .
\end{equation*}

So, we have

\begin{equation*}
\begin{split}
\int_{T_{k-1}}^{1}|\int_{\mathbb{R}^3} \frac{v_{k}}{|u|}u \nabla P_{k1} dx |ds\\
& = \int_{T_{k-1}}^{1}|\int_{\mathbb{R}^3} P_{k1} \nabla (\frac{v_{k}}{|u|}u)dx|ds\\
& \leqslant 3 \int_{T_{k-1}}^{1} \int_{\mathbb{R}^3}d_{k}|P_{k1}|\chi_{\{v_{k}>0\}}dx\,ds\\
& \leqslant 3 \|d_{k}\|_{L^2(Q_{k-1})} \|P_{k1}\|_{L^6(Q_{k-1})} \|\chi_{\{v_{k}>0\}}\|_{L^3(Q_{k-1})}\\
& \leqslant 3 (2^{\frac{1}{2}}) \|d_{k-1}\|_{L^2(Q_{k-1})} 3C_{6} 2^{\frac{10k}{9}} C^{\frac{1}{3}}
U_{k-1}^{\frac{5}{9}}\\
& \leqslant 9 (2^{\frac{1}{2}}) C_{6} C^{\frac{1}{3}} U_{k-1}^{\frac{1}{2}}
2^{\frac{10k}{9}}U_{k-1}^{\frac{5}{9}} .
\end{split}
\end{equation*}

That is, we have the following conclusion that

\begin{equation}
\int_{T_{k-1}}^{1}|\int_{\mathbb{R}^3} \frac{v_{k}}{|u|} u \nabla P_{k1} dx|ds
\leqslant C 2^{\frac{10k}{9}} U_{k-1}^{\frac{19}{18}} .
\end{equation}

Next, we would like to estimate the term
$\int_{T_{k-1}}^{1}|\int_{\mathbb{R}^3} \frac{v_{k}}{|u|}u \nabla P_{k2}dx|ds$. First, we recall that, by the
very definition of $P_{k2}$ ,we have

\begin{equation*}
P_{k2} = \sum R_{i}R_{j} \{ \frac{v_{k}}{|u|} u_{i} \frac{v_{k}}{|u|} u_{j} \} .
\end{equation*}

, in which $R_{i}$ , $R_{j}$ etc are the Riesz's Transforms. Hence, we have

\begin{equation*}
\nabla P_{k2} = \sum R_{i} R_{j} \{ \nabla (\frac{v_{k}}{|u|} u_{i} \frac{v_{k}}{|u|} u_{j} ) \} .
\end{equation*}

Now, we notice that
\begin{equation*}
\begin{split}
|\nabla (\frac{v_{k}}{|u|} u_{i} \frac{v_{k}}{|u|} u_{j})|
\leqslant |\nabla ( \frac{v_{k}}{|u|}u_{i})| |\frac{v_{k}}{|u|} u_{j}| +
\frac{v_{k}}{|u|} |u_{i}| |\nabla (\frac{v_{k}}{|u|} u_{j})|\\
& \leqslant 3d_{k} v_{k} + v_{k} (3d_{k})\\
& = 6v_{k}d_{k} .
\end{split}
\end{equation*}

So, by applying the Riesz's Theorem in the theory of Singular integral, we have

\begin{equation*}
\begin{split}
\|\nabla P_{k2}\|_{L^{\frac{3}{2}}(Q_{k-1})}
\leqslant C_{\frac{3}{2}} \|v_{k} d_{k}\|_{L^{\frac{3}{2}}(Q_{k-1})}\\
& \leqslant C_{\frac{3}{2}}\{ (\int_{Q_{k-1}}v_{k}^{6})^{\frac{1}{4}}
(\int_{Q_{k-1}}d_{k}^2)^{\frac{3}{4}}\}^{\frac{2}{3}}\\
& = C_{\frac{3}{2}} (\int_{Q_{k-1}}v_{k}^6)^{\frac{1}{6}} (\int_{Q_{k-1}} d_{k}^2)^{\frac{1}{2}}\\
& \leqslant C_{\frac{3}{2}}\|u\|_{L^6(Q_{0})} \|d_{k}\|_{L^2(Q_{k-1})}\\
& \leqslant C_{\frac{3}{2}} (\frac{1}{A})^{\frac{1}{6}} (2)^{\frac{1}{2}} \|d_{k-1}\|_{L^2(Q_{k-1})}\\
& \leqslant 2^{\frac{1}{2}} C_{\frac{3}{2}} U_{k-1}^{\frac{1}{2}} .
\end{split}
\end{equation*}

So, by applying the generalized Holder's inequality with exponents $\frac{10}{3}$, $30$, $\frac{3}{2}$
to the terms $v_{k}$, $\chi_{\{v_{k}>0\}}$, $\nabla P_{k2}$ respectively, we yield

\begin{equation*}
\begin{split}
\int_{T_{k-1}}^{1}   |\int_{\mathbb{R}^3}\frac{v_{k}}{|u|}u\nabla P_{k2} dx| dt\\
& \leqslant \int_{Q_{k-1}} v_{k} \chi_{\{v_{k}>0\}} |\nabla P_{k2}| dx\,dt\\
& \leqslant \|v_{k}\|_{L^{\frac{10}{3}}(Q_{k-1})} \|\chi_{\{v_{k}>0\}}\|_{L^{30}(Q_{k-1})}
\|\nabla P_{k2}\|_{L^{\frac{3}{2}}(Q_{k-1})}\\
& \leqslant U_{k-1}^{\frac{1}{2}} 2^{\frac{k}{9}} C^{\frac{1}{30}} U_{k-1}^{\frac{5}{90}} 2^{\frac{1}{2}}
C_{\frac{3}{2}} U_{k-1}^{\frac{1}{2}}\\
& = C 2^{\frac{k}{9}} U_{k-1}^{1+\frac{5}{90}}\\
& = C 2^{\frac{k}{9}} U_{k-1}^{\frac{19}{18}} .
\end{split}
\end{equation*}

That is, we have

\begin{equation}
\int_{T_{k-1}}^{1} | \int_{\mathbb{R}^3} \frac{v_{k}}{|u|} u \nabla P_{k2} dx  | dt
\leqslant C 2^{\frac{k}{9}} U_{k-1}^{\frac{19}{18}} .
\end{equation}
So, by combining inequalities , we yield

\begin{equation*}
\begin{split}
U_{k} \leqslant 2^{\frac{7k}{3}}C U_{k-1}^{\frac{5}{3}}
+ \int_{T_{k-1}}^{1} |\int_{\mathbb{R}^3} \frac{v_{k}}{|u|}u\nabla P dx| dt\\
& \leqslant 2^{\frac{7k}{3}}CU_{k-1}^{\frac{5}{3}} + C2^{\frac{10k}{9}}U_{k-1}^{\frac{19}{18}}
+ C2^{\frac{k}{9}}U_{k-1}^{\frac{19}{18}}\\
& \leqslant 2^{\frac{7k}{3}}CU_{k-1}^{\frac{19}{18}} .
\end{split}
\end{equation*}

That is, we will have the result that

\begin{equation*}
U_{k} \leqslant C2^{\frac{7k}{3}}U_{k-1}^{\frac{19}{18}} ,
\end{equation*}

for any $k\geqslant 1$.

\section{Proof of proposition \ref{third}}

Now, we would like to establish proposition~\ref{third} on the
foundation of proposition~\ref{first}. To begin, let $C^*$ be the
positive universal constant occuring in proposition~\ref{first}.
First, let show the proposition in the special case $\lambda=2$. We
chose $T$ to be an arbritary chosen positive number greater than
$2$, and let $u$ be a solution of the Navier-Stokes equation on
$(0,\infty)\times \mathbb{R}^3$. In the case in which
 $u$  satisfies the condition that $\int_{0}^{T}\int_{\mathbb{R}^3}|u|^6dx\,ds \leqslant (C^*)^6$, we define
 the function $u^*$ by $u^*(s,x) = u(s+(T-1), x)$,which can be regarded to be another solution of the
 Navier-Stokes equation on $[-1,1]\times \mathbb{R}^3$ satsifying

\begin{equation*}
\int_{-1}^{1}\int_{\mathbb{R}^3}|u^*|^6dx\,ds = \int_{T-2}^{T}\int_{\mathbb{R}^3}|u|^6dx\,ds
\leqslant \int_{0}^{T}\int_{\mathbb{R}^3}|u|^6dx\,ds \leqslant (C^*)^6 .
\end{equation*}

Hence, we have $\|u^*\|_{L^6([-1,1]\times \mathbb{R}^3)} \leqslant C^*$. So, it follows from the
conclusion of proposition~\ref{first}
 that $\|u(T, \cdot)\|_{L^{\infty}(\mathbb{R}^3)} = \|u^*(1, \cdot)\|_{L^{\infty}(\mathbb{R}^3)}\leqslant 1$. \\
So, the above argument shows that

\begin{itemize}
\item we have $\|u(T, \cdot)\|_{L^{\infty}(\mathbb{R}^3)}\leqslant 1$, if $T>2$, and $u$
is a solution of the Navier-Stokes equation satisfying $\int_{0}^{T}\int_{\mathbb{R}^3}|u|^6dx\,ds \leqslant (C^*)^6$.
\end{itemize}

Next, we also need to deal with the case in which the solution $u$
satisfies the condition that
$\int_{0}^{T}\int_{\mathbb{R}^3}|u|^6dx\,ds > (C^*)^6$. In this
case, let us consider the function $u_{\varepsilon}$ defined by
$u_{\varepsilon}(t,x) = \varepsilon u(\varepsilon^2 t, \varepsilon
x)$, in which $\varepsilon > 0$ is arbritary. Then, by applying the
change of variable formula, it is easy to see that

\begin{equation*}
\int_{0}^{\frac{T}{\varepsilon^2}}\int_{\mathbb{R}^3}|u_{\varepsilon}|^6dx\,ds
= \varepsilon \int_{0}^{T}\int_{\mathbb{R}^3}|u|^6dx\,ds .
\end{equation*}

So, by taking $\varepsilon = (C^*)^6\cdot
\{2\int_{0}^{T}\int_{\mathbb{R}^3}|u|^6dx\,ds\}^{-1}$, we yield

\begin{equation*}
\int_{0}^{\frac{T}{\varepsilon^2}}\int_{\mathbb{R}^3}|u_{\varepsilon}|^6dx\,ds
= \frac{(C^*)^6}{2} < (C^*)^6 .
\end{equation*}

The last inequality signifies that the solution $u_{\varepsilon}$
falls back to the first case in this discussion. Hence, it follows
directly form the conclusion we made for the frist case that
$u_{\varepsilon}$ must satisfies
$\|u_{\varepsilon}(\frac{T}{\varepsilon^2} , \cdot
)\|_{L^{\infty}(\mathbb{R}^3)} \leqslant 1$. So, we eventually have

\begin{equation*}
\|u(T, \cdot)\|_{L^{\infty}(\mathbb{R}^3)} = \frac{1}{\varepsilon}
\|u_{\varepsilon}(\frac{T}{\varepsilon^2}
,\cdot)\|_{L^{\infty}(\mathbb{R}^3)}
 \leqslant \frac{1}{\varepsilon} = \frac{2}{(C^*)^6} \int_{0}^{T}\int_{\mathbb{R}^3}|u|^6dx\,ds .
\end{equation*}

As a result, by all the discussion we made as above, we conclude that, no matter in which case, we always have the
following inequality to be valid for any $T>2$, and any solution $u$ of the Navier-Stokes equation on
$(0,\infty )\times\mathbb{R}^3$

\begin{equation*}
\|u(T, \cdot )\|_{L^{\infty}(\mathbb{R}^3)} \leqslant A \{ 1 +
\int_{0}^{T}\int_{\mathbb{R}^3}|u|^6dx\,ds \},
\end{equation*}
 where $A$ is the universal constant defined by $A = max\{ 1, \frac{2}{(C^*)^6}\}$. This gives the proof of Proposition \ref{third}
 in the special case $\lambda=2$.

Next, let $\lambda$ be a fixed positive number satisfying $0 <
\lambda < 2$. As usual, let $u$ be a solution of the Navier-Stokes
equation on $(0, \infty)\times \mathbb{R}^3$. Now, let us consider
the function $w$ which is defined by
\begin{equation*}
w(t,x) = (\frac{\lambda}{2})^{\frac{1}{2}} u(\frac{\lambda}{2} t , (\frac{\lambda}{2})^{\frac{1}{2}} x) .
\end{equation*}

Then, by applying the above case to $w$, we have the following
estimation, which is valid for any $T > \lambda$.

\begin{equation*}
\begin{split}
\|u(T,\cdot )\|_{L^{\infty}(\mathbb{R}^3)}
\leqslant (\frac{2}{\lambda})^{\frac{1}{2}}\|w(\frac{2T}{\lambda} , \cdot )\|_{L^{\infty}(\mathbb{R}^3)}\\
& \leqslant    (\frac{2}{\lambda})^{\frac{1}{2}}A \{ 1 +
\int_{0}^{\frac{2T}{\lambda}}\int_{\mathbb{R}^3}|w|^6dx\,ds\}\\
& \leqslant  (\frac{2}{\lambda})^{\frac{1}{2}} A \{ 1 +
(\frac{\lambda}{2})^{\frac{1}{2}}\int_{0}^{T}\int_{\mathbb{R}^3}|u|^6dx\,ds \}\\
& \leqslant (\frac{2}{\lambda})^{\frac{1}{2}} A \{ 1 + \int_{0}^{T}\int_{\mathbb{R}^3}|u|^6dx\,ds \} .
\end{split}
\end{equation*}

This gives proposition~\ref{third}, where the universal constant
$A_{\lambda}$ is chosen to be $A_{\lambda} =
(\frac{2}{\lambda})^{\frac{1}{2}} A$.

\section{establishment of Theorem 1}

Finally, we are now ready to establish the conclusion of
Theorem~\ref{main} on the foundation of proposition~\ref{third}. We
make use of the following result  due to Kato~\cite{Kato} (see also
the book of Lemari\'e-Rieusset \cite{Rieusset}).

\begin{thm}
Let $p>3$. Then, for any given initial datum $u_{0} \in
L^p(\mathbb{R}^3)$ satisfying $div (u_{0}) = 0$, there exists a
positive $T^*$ and a unique weak solution $u \in C ([0,T^*) ;
L^p(\mathbb{R}^3))$ for the Navier-Stokes equation on $(0,T^*)\times
\mathbb{R}^3$ so that $u(0, \cdot) = u_{0}$. This solution is then
smooth on $(0,T^*)\times \mathbb{R}^3$. In addition, such a unique
solution will also satisfies the extra condition that $u(t, \cdot)
\in C_{0}(\mathbb{R}^3)$, for all $t \in (0, T^*)$.
\end{thm}

To begin, let $u$ be a weak solution of the Navier-Stokes equation
on $(0, \infty)\times \mathbb{R}^3$ satisfying the condition that
$\int_{0}^{\infty}\int_{\mathbb{R}^3} \frac{|u|^5}{log (1+ |u| )}
dx\,ds < \infty$. Then, by using the elementary inequality $log(1+t)
\leqslant t$, which is valid for all $t \geqslant 0$, we can deduce
at once that

\begin{equation*}
\int_{0}^{\infty}\int_{\mathbb{R}^3} |u|^4 dx\,ds \leqslant
\int_{0}^{\infty}\int_{\mathbb{R}^3} \frac{|u|^5}{log(1 + |u|)}dx\,ds < \infty .
\end{equation*}

Now, let $\lambda \in (0, 2)$ to be arbritary chosen and fixed. Since $\int_{0}^{\infty}\int_{\mathbb{R}^3}
|u|^4 dx\,ds < \infty$, it follows that the quantity $\int_{\mathbb{R}^3}|u(t,x)|^4 dx$ must be finite for
almost every $t \in (0, \infty)$. So, with respect to $\lambda$, we can choose some $\tau_{0}$ with $0 <
\tau_{0} < \lambda$ in such a way that $\int_{\mathbb{R}^3}|u(\tau_{0} , x)|^4 dx < \infty$, or equivalently
$u(\tau_{0} , \cdot ) \in L^4(\mathbb{R}^3)$. So, by using a simple shifting technique, we  may apply the
Kato's Theorem quoted as above to deduce that there exists some positive constant $T^* > \tau_{0}$ so that
our weak solution  $u$ is smooth on $(\tau_{0} , T^*) \times \mathbb{R}^3$, and that $u(t, \cdot) \in
C_{0}(\mathbb{R}^3)$, for every $t$ with $\tau_{0} < t < T^*$. Hence, we know, in particular, that our weak
solution $u$ must be lying in the space $L_{loc}^{\infty} (\tau_{0} , T^* ; L^{\infty}(\mathbb{R}^3))$. Now,
for some technical purpose, we would like to pick up two numbers $\tau_{1}$ and $\tau_{2}$ which verify the
condition that $\tau_{0} < \tau_{1} < \tau_{2} < min \{\lambda , T^* \}$. Once $\tau_{1}$ and $\tau_{2}$ are
chosen, they will be fixed. Now, from our original weak solution $u$, we can construct another weak solution
$v$ by requiring that $v(t,x) = u(t+ \tau_{1} , x)$. Now, by applying the conclusion of
proposition~\ref{third} to the weak  solution $v$ and the number $\tau_{2} - \tau_{1}$, we can at once deduce
that we have the following inequality

\begin{equation*}
\|v(t, \cdot )\|_{L^{\infty}(\mathbb{R}^3)} \leqslant
A \{ 1 + \int_{0}^{t}\int_{\mathbb{R}^3} |v|^6 dx\,ds\} ,
\end{equation*}

to be valid for all $t > \tau_{2} - \tau_{1}$ , in which $A$ is some universal constant depending only on
$\tau_{2} - \tau_{1}$ . However, this means the same as saying that we have the following
inequality

\begin{equation*}
\|u(t+ \tau_{1} , \cdot )\|_{L^{\infty}(\mathbb{R}^3)} \leqslant
A \{ 1 + \int_{\tau_{1}}^{t + \tau_{1}}\int_{\mathbb{R}^3} |u|^6 dx\,ds \} ,
\end{equation*}

which is valid for all $t > \tau_{2} - \tau_{1}$. Hence , it follows that we can make the following
conclusion

\begin{itemize}
\item for every $t > \tau_{2}$, we have $\|u(t, \cdot)\|_{L^{\infty}(\mathbb{R}^3)} \leqslant A \{ 1 +
\int_{\tau_{1}}^{t}\int_{\mathbb{R}^3} |u|^6 dx\,ds \}$ , in which $A$ is some universal constant depending
only on $\tau_{2} - \tau_{1}$.
\end{itemize}

At this stage, we are ready to apply the Gronwall's argument in the
theory of ordinary differential equations as follow. For this
purpose, we take $\psi (t) = t \cdot log (1 + t)$, which is a
strictly increasing positive valued function on $(0, \infty )$
satisfying the condition that

\begin{equation*}
\int_{1}^{\infty} \frac{1}{\psi (t)} dt = \infty .
\end{equation*}

Then, it follows from our last inequality that

\begin{equation*}
\begin{split}
\|u(t, \cdot)\|_{L^{\infty}(\mathbb{R}^3)}\\ &\leqslant
 A \{ 1 + \int_{\tau_{1}}^{t}\int_{\mathbb{R}^3} \psi (|u|) \frac{|u|^5}{log (1+|u|)}dx\,ds\} \\
&\leqslant A \{ 1 + \int_{\tau_{1}}^{t} \psi (\|u\|_{L^{\infty}(\mathbb{R}^3)})
\int_{\mathbb{R}^3}\frac{|u|^5}{log (1+ |u|)} dx\,ds \}  ,
\end{split}
\end{equation*}

which is valid for all $t > \tau_{2}$. \\
Next ,we put $F(t) = \|u(t, \cdot )\|_{L^{\infty}(\mathbb{R}^3)}$. Then, the above inequality can be
rewritten as

\begin{equation}
F(t) \leqslant A \{  1 + \int_{\tau_{1}}^{t} \psi (F(s)) G(s) ds \},
\end{equation}

for all $t > \tau_{2}$, where $G$ is the function defined by
$G(s) = \int_{\mathbb{R}^3} \frac{|u|^5}{log (1+ |u|)} dx$.
Furthermore, we notice that by the hypothesis of Theorem~\ref{main}, the function $G$ must
satisfies the condition that

\begin{equation*}
\int_{0}^{\infty} G(s) ds = \int_{0}^{\infty}\int_{\mathbb{R}^3}\frac{|u|^5}{log (1 + |u|)}dx\,ds
< \infty .
\end{equation*}

Here, for the sake of convenience, we define

\begin{equation*}
H(t) = A \{ 1 + \int_{\tau_{1}}^{t} \psi (F(s)) G(s) ds\} ,
\end{equation*}

for all $t > \tau_{1}$. Then, our last inequality can be rewritten as

\begin{itemize}
\item $F(t) \leqslant H(t)$ , for all $t > \tau_{2}$.
\end{itemize}

Since $\psi$ is a strictly increasing  positive valued function on $(0 , \infty )$, it follows at once that

\begin{equation*}
\frac{dH}{dt} = A \psi (F(t))G(t) \leqslant A \psi (H(t))G(t) ,
\end{equation*}

which is valid for all $t > \tau_{2}$. That is, we have the fact that

\begin{itemize}
\item for every $t > \tau_{2}$ , we have $\frac{dH}{dt} \leqslant A \psi (H(t))G(t)$.
\end{itemize}

As a result, by taking integration in time over the interval $(\tau_{2} , t)$, for $t > \tau_{2}$, it follows at
once that

\begin{equation*}
\Psi (H(t)) - \Psi (H(\tau_{2})) \leqslant A \int_{\tau_{2}}^{t} G(s) ds ,
\end{equation*}

for all $t > \tau_{2}$, in which $\Psi$ is the function defined by
$\Psi(y) = \int_{A}^{y} \frac{1}{\psi (y)} dy$. Hence, we can deduce
that

\begin{itemize}
\item for every $t > \tau_{2}$, we have $\Psi (H(t)) \leqslant \Psi (H(\tau_{2})) + A \int_{\tau_{2}}^{t} G(s)
ds$ .
\end{itemize}

At this stage, in order to complete the Gronwall's argument successfully, we definitely need to show that
$H(\tau_{2})$ is finite. To achive this, let us recall that we have already used the Kato's Theorem to deduce
that our original weak solution $u$ must satisfies $u \in L_{loc}^{\infty} (\tau_{0} , T^* ;
L^{\infty}(\mathbb{R}^3))$, and this at once tells us that
$\|u\|_{L^{\infty}([\tau_{1} , \tau_{2}]\times \mathbb{R}^3)} =  sup_{t\in [\tau_{1} , \tau_{2} ]} F(t) <  +\infty$ ,
because of the fact that $0 < \tau_{0} < \tau_{1} < \tau_{2} < min \{\lambda , T^* \}$. Hence, it follows
immediately that

\begin{equation*}
H(\tau_{2}) \leqslant A \{1 + \psi (\|u\|_{L^{\infty}([\tau_{1} ,\tau_{2}]\times \mathbb{R}^3)}) \int_{\tau_{1}}^{\tau_{2}}  G(s) ds \}
< +\infty    .
\end{equation*}

So, we can now combine $H(\tau_{2}) < \infty$, and $\int_{0}^{\infty}G(s)ds < \infty$ to deduce that

\begin{itemize}
\item for every $t > \tau_{2}$, $\Psi (H(t)) \leqslant \Psi (H(\tau_{2})) + \int_{\tau_{2}}^{t} G(s) ds <
\infty$.
\end{itemize}

That is, we now know that $\Psi (H(t))$ must be finite, for every $t > \tau_{2}$. Since
$\int_{A}^{+\infty}\frac{1}{\psi (y)} dy = +\infty$, this will force us to admit that $H(t) < \infty$, for all
$t > \tau_{2}$. Hence, we eventually have the conclusion

\begin{itemize}
\item for every $t > \tau_{2}$, we have $\|u(t, \cdot )\|_{L^{\infty}(\mathbb{R}^3)} = F(t) \leqslant H(t) <
\infty$ . So, in particular, we now know also that $u \in L_{loc}^{\infty}(\tau_{2} , \infty ;
L^{\infty}(\mathbb{R}^3))$.
\end{itemize}

Since our weak solution $u$ now satisfies the condition $u \in L_{loc}^{\infty} (\tau_{2} , \infty ;
L^{\infty}(\mathbb{R}^3))$, by
applying the famous result of Serrin~\cite{Serrin} that we mentioned in the introduction with the case
in which $p = q = \infty$, $u\in L_{loc}^{\infty} ( (\tau_{2} , \infty )\times \mathbb{R}^3 )$
immediately implies that $u \in C^{\infty} ( (\tau_{2} , \infty )\times \mathbb{R}^3)$, and hence we have
the conclusion that $u$ must be smooth on $(\lambda ,\infty) \times \mathbb{R}^3$ (notice that $\tau_{2} <
\lambda $). Since $\lambda \in (0, 2)$ is arbritary chosen in the above argument, we can finally deduce
that any weak solution $u$ satisfying the hypothesis of Theorem~\ref{main} must be smooth on $(0,
\infty)\times \mathbb{R}^3$.

\end{document}